\documentclass[11pt]{amsart}
\usepackage[active]{srcltx}

\usepackage{amsmath, amsfonts,amsthm,times,graphics,color}
\newcommand{\vertiii}[1]{{\left\vert\kern-0.25ex\left\vert\kern-0.25ex\left\vert #1
    \right\vert\kern-0.25ex\right\vert\kern-0.25ex\right\vert}}
 \makeatletter
\renewcommand*\subjclass[2][2000]{%
  \def\@subjclass{#2}%
  \@ifundefined{subjclassname@#1}{%
    \ClassWarning{\@classname}{Unknown edition (#1) of Mathematics
      Subject Classification; using '1991'.}%
  }{%
    \@xp\let\@xp\subjclassname\csname subjclassname@#1\endcsname
  }%
}
 \makeatother
\usepackage{enumerate,amssymb,  mathrsfs,yhmath}

\newtheorem{theorem}{Theorem}[section]
\newtheorem{lemma}[theorem]{Lemma}
\newtheorem*{lemma*}{Lemma}
\newtheorem{proposition}[theorem]{Proposition}
\newtheorem{corollary}[theorem]{Corollary}

\def\1ton{1,2,\ldots,n}

\usepackage{amssymb}
\usepackage{amsthm}
\usepackage{mathrsfs, amsfonts, amsmath}
\usepackage{graphicx}

\theoremstyle{definition}

\newtheorem{conjecture}[theorem]{Conjecture}

\theoremstyle{remark}
\newtheorem{remark}[theorem]{Remark}

\numberwithin{equation}{section}

\renewcommand{\imath}{i} 

\def\XXint#1#2#3{{\setbox0=\hbox{$#1{#2#3}{\int}$}
\vcenter{\hbox{$#2#3$}}\kern-.5\wd0}}

\def\ge{\geqslant}
\begin{document}

\title{Contraction property of differential operator on Fock space}

\author{David Kalaj}
\address{University of Montenegro, Faculty of Natural Sciences and
Mathematics, Cetinjski put b.b. 81000 Podgorica, Montenegro}
\email{davidk@ucg.ac.me}


\begin{abstract}
In the recent paper, \cite{tilli} Nicola and Tilli proved the  Faber-Krahn inequality, which for $p=2$,  states the following. If $f\in\mathcal{F}_\alpha^2$ is an entire function  from the corresponding Fock space, then $$\frac{1}{\pi}\int_{\Omega} |f(z)|^2 e^{-\pi |z|^2} dx dy \le (1-e^{-|\Omega|}) \|f\|^2_{2,\pi}.$$ Here $\Omega$ is a domain in the complex plane and $|\Omega|$ is its Lebesgue measure. This inequality is sharp and equality can be attained. We prove the following sharp inequality $$\int_{\Omega} \frac{|f^{(n)}(z)|^2e^{-\pi |z|^2}}{\pi^n n ! L_n(-\pi |z|^2)}dxdy \le (1-e^{-(n+1)|\Omega|})\|f\|^2_{2,\pi},$$ where $L_n$ is Laguerre polynomial, and $n\in\{0,1,2,3,4\} $. For $n=0$ it coincides with the result of Nicola and Tilli.

\end{abstract}
\maketitle
\tableofcontents
\sloppy

\maketitle
\section{Introduction}

Let $\mathbb{C}$ be the complex plane. Denote by $dA(z)(=dxdy)$ the Lebesgue measure on the complex plane. Throughout  the paper,   we consider the Gaussian-probability measure
$$d\mu_{\alpha}(z)=\frac{\alpha}{\pi}e^{-\alpha|z|^{2}}dA(z),$$
where $\alpha$ is a positive parameter. Also let $dA_\alpha(z)=\frac{\alpha}{\pi} dA(z)$.

For $1\leq p<\infty,$ let $L^{p}(\mathbb{C},d{\mu}_\alpha)$ denote the space of all Lebesgue measurable functions $f$ on $\mathbb{C}
$ such that
$$\|f\|_{p,\alpha}^{p}=\frac{p\alpha}{2\pi}\int_{\mathbb{C}}|f(z)|^{p}e^{-\frac{p\alpha|z|^{2}}{2}}dA(z)<\infty.$$

 The Segal--Bargmann space  also known as the Fock space, denoted by $\mathcal{F}_{\alpha}^{2},$ consists of all entire functions $f$ in $L^{2}(\mathbb{C},d\mu_{\alpha}).$ For any $p\geq 1$, $F_{\alpha}^{p}$ is a closed subspace of $L^{p}(\mathbb{C},d\mu_{\alpha}).$ Therefore $F_{\alpha}^{p}$ is a Banach space.

We refer to  the book \cite{zhu} for a good setting of the Fock space on the plane and the papers \cite{foland}, \cite{hall} and \cite{janson} for a higher-dimensional setting.

Since  $L^{2}(\mathbb{C},d\mu_{\alpha})$ is a Hilbert space with the inner product $$\left<f,g\right>_\alpha =\int_{\mathbb{C}}f(z)\overline{g(z)}d\mu_{\alpha}(z),$$
the Fock space $\mathcal{F}_{\alpha}^{2}$ as its closed subspace determines a natural orthogonal projection $P_{\alpha}: L^{2}(\mathbb{C},d\mu_{\alpha}) \rightarrow \mathcal{F}_{\alpha}^{2}.$

It can be shown (see \cite{zhu}) that $P_{\alpha}$ is an integral operator induced by the reproducing kernel $$K_{\alpha}(z,w)=e^{\alpha z\bar{w}}.$$
More precisely, $$P_{\alpha}f(z)=\int_{\mathbb{C}}K_{\alpha}(z,w)f(w)d\mu_{\alpha}(w), f\in L^{2}(\mathbb{C},d\mu_{\alpha}),$$ and in particular
\begin{equation}
\label{hilbert}f(z)=\int_{\mathbb{C}}K_{\alpha}(z,w)f(w)d\mu_{\alpha}(w), f\in \mathcal{F}_{\alpha}^{2}.\end{equation}

Laguerre polynomial is defined by
$$L_n(x)=\sum_{k=0}^n \binom{n}{k}\frac{(-1)^k}{k!}x^k.$$ For $n=0$, $L_0(x)=1$.

Le $F={_1}F_1$ be the Kummer  function defined by $$F[a,b,z]=\sum_{k=0}^\infty \frac{(a)_k}{(b)_k k!} z^k,$$ where $(c)_k=c(c+1)\cdot \dots \cdot (c+k-1)$ is denoted the shifter factoriel. The connection between Kummer function and Laguerre polynomial is given by $$F(1+n,1,r) = e^{r}L_n(-r).$$

For $\alpha>0$ and  $n\ge 0$ an integer, we define the Hilbert  space of entire functions $\mathcal{F}^2_{n,\alpha}=L^2(\mathbb{C}, d\mu_{n,\alpha})$, where 
$$d\mu_{n,\alpha}(z)=\frac{1}{n!\alpha^n L_n(-\alpha |z|^2)} d\mu_\alpha(z).$$ In this paper we  consider the differential operator  $\mathcal{D}_n [f](z)=f^{(n)}(z),$ and show that it maps $\mathcal{F}^2_{\alpha}$ into  $\mathcal{F}^2_{n,\alpha}$. Moreover, we show that it is a contraction satisfying the local contraction property.

We start with a result of Haslinger, author, and Vujadinovic in \cite{cmft}, which can be stated as the contraction property of differential operator from $\mathcal{F}^2_{\alpha}$ into  $\mathcal{F}^\infty_{n,\alpha}$, where the last space is defined in a standard fashion.  Note that the space $F_{n,\alpha}^{\infty}$ is defined to be the space of all entire functions $f$ such that
 $$\|f\|_{\infty,n,\alpha}:=\alpha^{-n/2}(n!)^{-1/2}\sup\{|f(z)|{e^{-\frac{\alpha|z|^{2}}{2}}}{L_n^{-1/2}(-\alpha|z|^2)}:z\in\mathbb{C}\}<\infty.$$ 
\begin{proposition}
\label{hilbert11} Let $f\in \mathcal{F}_{\alpha}^{2}.$  Then we have the  sharp point-wise inequality
$$|f^{(n)}(z)|\leq  e^{\alpha |z|^2/2}\sqrt{\alpha^n n! L_n(-\alpha |z|^2) }\|f-T_{n}(f)\|_{2,\alpha},$$
where $n\in \mathbf{N}_0$ and $T_{n}(f)(w)=\sum_{k=0}^{n-1}\frac{f^{k}(0)}{k!}w^{k}$,
and  $L_n$ is Laguerre polynomial.
The extremal functions are of the form
$$f(w)=\sum_{k=0}^{n-1}a_{k}w^{k}+\alpha^{n}e^{\alpha \bar{z} w}w^{n}.$$
In particular, for $n=1$ we get the following best estimates
\begin{equation}\label{est2}
\begin{split}|f'(z)|\le
\sqrt{\alpha  \left(1+\alpha |z|^2\right)}e^{\alpha  |z|^2/2} \|f\|_{2,\alpha},\ \ \  z\in\mathbb{C}.\end{split}
\end{equation}
In other words we have the sharp inequality $$\|f^{(n)}\|_{\infty, n,\alpha}\le \|f\|_{2,\alpha}.$$
\end{proposition}

Proposition~\ref{hilbert11} for $n=0$ coincides with a corresponding result \cite[Theorem~2.7]{zhu} or \cite[Theorem~3.4.2]{groc}, and that result was crucial for proving the following important result in \cite{tilli}, by Nicola and Tilli \begin{theorem}\label{nicola}
 For $f\in \mathcal{F}_\alpha^2$, and a domain   $\Omega$ with a finite measure,  we have
 $$\int_{\Omega} {|f(z)|^2e^{-\pi |z|^2}}dxdy \le (1-e^{-|\Omega|})\|f\|^2_{2,\alpha}.$$ This result is sharp and equality is attained for certain balls and certain specific functions related to extremal functions in Proposition~\ref{hilbert11}.
\end{theorem}
Theorem~\ref{nicola} is further extended in \cite{tilli} for $p\ge 2$, and proved local  Lieb's uncertainty inequality for the STFT (short-time  Fourier transforms see \cite[Theorem~5.2]{tilli}), by using Lieb's uncertainty inequality for the STFT. Theorem~\ref{nicola} proves a conjecture of Abreu and Speckbacher \cite[Conjecture~1]{as} (for  $p = 2$). We extend Theorem~\ref{nicola} for higher derivatives. We  believe that our results can be also applied to STFT, especially because of the Laguerre connection \cite[Chapter~1]{foland}. By using the method developed in \cite{tilli}, Kulikov in \cite{kulik} proved some conjectures for Bergman and Hardy space of holomorphic mappings. We also use the method developed in \cite{tilli}.

In this paper, we consider the following conjecture
\begin{conjecture}\label{c1}
 For $f\in \mathcal{F}_\alpha^2$, and a domain   $\Omega$ with a finite measure,  we have 
 
 $$\int_{\Omega} {|f^{(n)}(z)|^2}d\mu_{n,\alpha}(z) \le (1-e^{-(n+1)|\Omega|})\|f\|^2_{2,\alpha},$$ where  $n\ge 0$.
\end{conjecture}
We  reduce Conjecture~\ref{c1} to  the following conjecture

 \begin{conjecture}\label{c2}
 For $f\in \mathcal{F}_\alpha^2$,  we have $$\|f^{(n)}\|^2_{2,n,\alpha}:=\int_{\mathbb{C}} {|f^{(n)}(z)|^2}d\mu_{n,\alpha}(z)\le \|f\|^2_{2,\alpha}.$$ 
\end{conjecture}

Further Conjecture~\ref{c2} is reduced to the following \begin{conjecture}\label{c3}
Let $k\ge n$ and $$a_k=\frac{ \Gamma[1+k]}{n!\Gamma[1+k-n]^2}\int_0^\infty \frac{e^{-r}r^{k-n}}{L_n(-r)} dr,$$ where $L_n$ is Laguerre polynomial, and $n\ge 0$.

Then $a_k<1$.

\end{conjecture}
Then we  prove Conjecture~\ref{c1}  for $n\le 4$. In other words, we prove
 \begin{theorem}\label{111}
 For $f\in \mathcal{F}_\alpha^2$, and a domain   $\Omega\subset \mathbb{C}$ with a finite Lebesgue measure $|\Omega|$,  we have $$\frac{\alpha}{\pi}\int_{\Omega} \frac{|f^{(n)}(z)|^2e^{-\alpha |z|^2}}{\alpha^n n ! L_n(-\alpha |z|^2)}dxdy \le (1-e^{-(n+1)|\Omega|})\|f\|^2_{2,\alpha},$$ where $L_n$ is Laguerre polynomial, and $0\le n\le 4$. The inequality is sharp but is never attained.
In particular for $n=1$, we have
$$\int_\Omega \frac{|f'(z)|^2 e^{-\alpha |z|^2}}{\alpha(1+\alpha|z|^2)}dA_\alpha(z)\le (1-e^{-2|\Omega|})\|f\|^2_{2,\alpha}.$$
In other words the differential operator 
$\mathcal{D}_n: \mathcal{F}_\alpha^2\to  \mathcal{F}^2_{n,\alpha}$,  is a contraction and satisfies local contraction property on sub-domains of complex plane.

\end{theorem}

\begin{remark}
Note that for $n=0$ the previous theorem is exactly the main result of Nicola and Tilli in \cite{tilli}. Note also that $f\in \mathcal{F}_\alpha^2 \not \Rightarrow f'\in \mathcal{F}_\alpha^2$ (see a counterexample in \cite{cmft}), so Laguerre polynomials are inevitable. 
\end{remark}

\begin{corollary}
Under conditions of Theorem~\ref{111} we have $$\int_\Omega \frac{|f'(z)|^2 e^{-\alpha |z|^2}}{\alpha(1+\alpha|z|^2)}dA_\alpha(z)\le (1-e^{-2|\Omega|})d^2_\alpha(f, \mathcal{C}),$$ where $\mathcal{C}$ is the space of all constants $d_\alpha$ is the induced metric in the Fock space $\mathcal{F}_\alpha^2$.
\end{corollary}

\section{Proof of main results}

We need the following corollary of Proposition~\ref{hilbert11}.

\begin{corollary}\label{prim}
  If $f\in F_\alpha^2$, then $$ \lim_{|z|\to \infty} \frac{|f^{(n)}(z)|}{\sqrt{\alpha^n n!L_n(-\alpha |z|^2)}}e^{-\alpha  |z|^2/2}=0.$$ and in particular, for $n=1$, we have $$ \lim_{|z|\to \infty} \frac{|f'(z)|}{\sqrt{\alpha(1+\alpha |z|^2)}}e^{-\alpha  |z|^2/2}=0.$$
\end{corollary}
\begin{proof} Let us prove for $n=1$ and note that the same proof works for every $n$. Since the space of polynomials is dense in the Fock space, for given $\epsilon>0$, there is a polynomial $P$ so that $\|f-P\|_{2,\alpha}\le \epsilon$. Then we have
$$\|f\|_{2,\alpha}\le \|f-P\|_{2,\alpha}+\|P\|_{2,\alpha}.$$
From \eqref{est2}, we get $$|f'(z)-P'(z)|\le  \sqrt{\alpha \left(1+\alpha |z|^2\right)}e^{\alpha  |z|^2/2} \|f-P\|_{2,\alpha}
\le \sqrt{\alpha  \left(1+\alpha |z|^2\right)}e^{\alpha  |z|^2/2}\epsilon.$$
On the other hand $$  \lim_{|z|\to \infty}\frac{|P'(z)|}{\sqrt{\alpha(1+\alpha |z|^2)}}e^{-\alpha  |z|^2/2}=0.$$
So
$$\limsup_{|z|\to \infty}\frac{|f'(z)|}{\sqrt{\alpha(1+\alpha |z|^2)}}e^{-\alpha  |z|^2/2 }\le \epsilon.$$ Thus $$\limsup_{|z|\to \infty} \frac{|f'(z)|}{\sqrt{\alpha(1+\alpha |z|^2)}}e^{-\alpha  |z|^2/2}=0$$ as it is stated.
\end{proof}

\begin{lemma}\label{kop} Let $$u(z) = \frac{|f^{(n)}(z)|^2e^{-\alpha |z|^2}}{\alpha^n n ! L_n(-\alpha |z|^2)}$$ and $A_t=\{z: u(z)>t\}$.
If $|\Omega|=s=\mu(t)$, then $$\int_{\Omega} u(z) dxdy\le \int_{A_{t}} u(z)dxdy.$$
\end{lemma}
\begin{proof}[Proof of Lemma~\ref{kop}]
As in \cite{tilli}, let $\Omega_1=\Omega\cap A_t$. Then $\Omega = \Omega_1\cup (\Omega\setminus \Omega_1)$ and $A_t = \Omega_1 \cap (A_1\setminus \Omega_1)$. We also have $$|(\Omega\setminus \Omega_1)|=|\Omega|-|\Omega_1|=|A_t|-|\Omega_1|=|A_t\setminus \Omega_1|.$$
If $z\in (\Omega\setminus \Omega_1)$, then $u(z)\le t$ and if $w\in A_t\setminus \Omega_1$ then $u(w)>t$.

Thus \[\begin{split}\int_{\Omega} u(z) dxdy&=\int_{\Omega_1} u(z) dxdy+\int_{\Omega\setminus \Omega_1} u(z) dxdy
\\& \le \int_{\Omega_1} u(z) dxdy+s\int_{\Omega\setminus \Omega_1}  dxdy
\\& =\int_{\Omega_1} u(z) dxdy+s\int_{A_t\setminus \Omega_1}  dxdy
\\& \le\int_{\Omega_1} u(z) dxdy+\int_{A_t\setminus \Omega_1} u(z) dxdy
\\&=\int_{A_{t}} u(z)dxdy.
\end{split}
\]
\end{proof}

\subsection{Proof Theorem~\ref{111} for $n=1$.}
We need the following general global estimate.
\begin{lemma}\label{import}
Assume that $\alpha>0$ and that $f\in F_\alpha^2$. Then $$\frac{1}{\pi}\int_{\mathbb{C}} \frac{|f'(z)|^2 e^{-\alpha |z|^2}}{1+\alpha|z|^2}dxdy\le \|f\|_{2,\alpha}^2.$$
\end{lemma}
\begin{proof} Since $f$ is an entire holomorphic functions, it has a Taylor representation:
$$f(w) =\sum_{k=0}^\infty a_k w^k.$$
Then
$$f'(w) = \sum_{k=1}^\infty k a_k w^{k-1}.$$
Thus $$\|f\|_{2,\alpha}^2=\sum_{k=0}^\infty |a_k|^2\frac{k!}{\alpha^k}.$$
Moreover \[\begin{split}\frac{\alpha}{\pi}\int_{\mathbb{C}} \frac{|f'(z)|^2 e^{-\alpha |z|^2}}{\alpha(1+\alpha|z|^2)}dxdy&=2\alpha\sum_{k=1}^\infty k^2|a_k|^2\int_0^\infty \frac{r^{2(k-1)+1}  e^{-\alpha r^2}}{\alpha(1+\alpha r^2)}dr
\\&=\sum_{k=1}^\infty k^2|a_k|^2  e E_k \frac{\Gamma(k)}{\alpha^k}\\&=\sum_{k=0}^\infty g_k|a_k|^2\frac{k!}{\alpha^k},\end{split}\] where $$E_k= E_k(1)=\int_1^\infty \frac{e^{-t}}{t^k}dt,$$ and
 $g_k= k e E_k.$ 
We need the following
\\
{\bf Claim.} \emph{For every $k\ge 1$ we have $g_k<1$. Moreover \begin{equation}\label{gk}\lim_{k\to\infty} g(k)=1.\end{equation}}
By using partial integration we get $$g_{k+1} = \frac{k+1}{k^2}(k-g_k).$$ Further $1/2\le g_1=0.596347<1$. Now  we prove by induction that $$\frac{k}{k+1}<g_k<1$$ for $k\ge 2$.
In order to do so let $$h(y) := \frac{k+1}{k^2}(k-y).$$ Then it is  clear $h(\frac{k}{k+1})=1$ and $$h(1)=\frac{k^2-1}{k^2}\ge \frac{k+1}{k+2}.$$
Thus the claim follows.

Since $e kE_k< 1$, we obtain that $$\frac{1}{\pi}\int_{\mathbb{C}} \frac{|f'(z)|^2 e^{-\alpha |z|^2}}{1+\alpha|z|^2}dxdy\le \|f\|_{2,\alpha}^2,$$ as claimed.

\end{proof}


\begin{proof}[Proof of Theorem~\ref{111} for $n=1$.]
Let $$u(z) = \frac{|f'(z)|^2 e^{-\alpha |z|^2}}{\alpha(1+\alpha|z|^2)}$$ and
define $\mu(t)=|A_t|.$  Then by Corollary~\ref{prim},  $A_t=\{z: u(z)>t\}$ is compactly supported in $\mathbb{C}$ and has a smooth boundary $\partial A_t=\{z: u(z)=t\}$ for almost every $t\in(0, \max u)$. Also $\lim_{|z|\to\infty} u(z)=0$ uniformly provided that $f\in L^2(\mathbf{C}, d\mu_\alpha)$,  where $$d\mu_\alpha=\frac{\alpha}{\pi}e^{-\alpha|z|^2}dxdy.$$  Let, $s=\mu(t)$ and define $$I(s)=\frac{\alpha}{\pi}\int_{A_t} u(z) dxdy.$$
Then as in \cite[Lemma~3.4]{tilli} we can prove that $I$ is smooth and $$I'(s) = \mu^{-1}(s).$$
In view of Lemma~\ref{kop}, we need to prove that $$I(s)\le (1-e^{-2|A_t|})\|f\|_{2,\alpha}.$$ Assume as we may that $\|f\|_{2,\alpha}=1$. Further we have $|A_t|=s$. Therefore we need to show that $$I(s)\le (1-e^{-2s}).$$
Then by using the Cauchy inequality to the length of $\partial A_t$ we obtain
$$|\partial A_t|^2\le \int_{\partial A_t} |\nabla u|^{-1} ds\cdot \int_{\partial A_t} |\nabla u|ds.$$
Further by \cite[Lemma~3.2]{tilli} we have $$\int_{\partial A_t} |\nabla u|^{-1} ds=-\mu'(t).$$ Let $\nu$ be the out-pointing unit normal at the smooth curve $\partial A_t$. Note that this curve is smooth for almost every $t>0$. This can be shown by using a similar approach as in \cite{tilli}.
Now $$ |\nabla u|=-\left<\nabla u, \nu\right >=-t\left<\nabla \log u, \nu\right >,$$ for $z\in \partial A_t=\{z: u(z)=t\}$.
Thus $$\int_{\partial A_t} |\nabla u|ds=-t\int_{\partial A_t} \left<\nabla \log u, \nu\right >ds .$$
By Green theorem we have $$\int_{\partial A_t} \left<\nabla \log u, \nu\right >ds=\int_{ A_t} \Delta \log u dA(z)$$
Since \begin{equation}\label{delu}\Delta \log u = -4\alpha\left(1+\frac{1}{\left(1+\alpha |z|^2\right)^2}\right),\end{equation} because $f'(z)$ is nonvanishing holomorphic function on $A_t$, 
we obtain
\[\begin{split}
   \int_{\partial A_t} \left<\nabla \log u, \nu\right >ds & =4t\alpha \int_{A_t} \left(1+\frac{1}{\left(1+\alpha |z|^2\right)^2}\right)dA(z) \\
     & =4t\pi  \int_{A_t} \left(1+\frac{1}{\left(1+\alpha |z|^2\right)^2}\right)dA_\alpha(z)\\ &\le 8 t \pi |A_t|.
\end{split}\]
Now we use the isoperimetric inequality for the plane: $$ 4 \pi |A_t|\le |\partial A_t|^2,$$ and obtain   $$4\pi \mu(t)\le (-\mu'(t))(8t\pi (\mu(t))).$$
The last inequality implies the inequality:  $$1 + 2 \frac{q(s)}{q'(s)}\le 0,$$ where $q(s)=\mu^{-1}(s)$. Because $q'(s)<0$, we obtain $q'(s)+2q(s)\ge 0$. Thus the function $$Q(s)=q(s)e ^{2s}$$ is increasing, because $$Q'(s) = e^{2s}(2q(s)+q'(s))\ge 0.$$

Therefore $$G(\sigma)=I(-1/2 \log \sigma)$$ is a convex function. Namely $$G'(\sigma) = -\frac{\mu^{-1}(-1/2 \log \sigma)}{2\sigma}=-e^{2s} q(s)=-Q(s),$$ where $s=e^{-2\sigma}$.  So $$G''(\sigma)=2Q'(s)e^{-2\sigma}>0.$$ Moreover $G$ satisfies the conditions  $G(0)\le 1$ and $G(1)=0$.
The first condition is satisfied because of Lemma~\ref{import}. Namely $$G(0) = I(+\infty)=\int_{A_0} u(z) dA_\alpha\le 1.$$
On the other hand
$$G(1) =  I(0)=\int_{A_\infty} u(z) dA_\alpha=0.$$

By using the convexity of $G$ we have $G(\sigma)\le (1-\sigma)$ and this implies the inequality 
\begin{equation}\label{iprim}I(s)\le (1-e^{-2s}).\end{equation}
To prove the sharpness, let $f(z) =z^m$. Then in view of \eqref{gk} we have $$\lim_{R\to \infty, m\to\infty} J(R,m)=1,$$ where
$$J(R,m)=\frac{\int_{|z|<R} \frac{|f'(z)|^2 e^{-\alpha |z|^2}}{\alpha(1+\alpha|z|^2)}dxdy}{(1-e^{-2\pi R^2})\|f\|^2_{2,\alpha}}.$$
\end{proof}

\subsection{The proof of Theorem~\ref{111} for $n\ge 2$}
The proof is just an imitation of the case $n=1$, so we omit some details. We begin by the following lemma
\begin{lemma}
Let $L(x,y) =\log F[1+n, 1, \alpha(x^2+y^2)]$, where $F$ is Kummer function.  Then $$\Delta L(x,y)\le 4\alpha (1+n).$$
\end{lemma}
\begin{proof}
First of all, we have $\Delta L(x,y)=g(t)$ where
$t=\left(x^2+y^2\right)$
and
\[\begin{split}&\frac{g(t)}{2 (1+n) \alpha}F\left[1+n,1,\alpha  t\right]^2=  -2 (1+n) \alpha  t F\left[2+n,2,\alpha  t\right]^2\\&+F\left[1+n,1,\alpha  t\right] \left(2 F\left[2+n,2,\alpha  t\right]+(2+n) \alpha  t F\left[3+n,3,\alpha  t\right]\right).
 \end{split}\]
We need to prove that $g(t)\le g(0)=4\alpha(1+n)$.
It is equivalent with the following $$k(t)=g(t)F\left[1+n,1,\alpha  t\right]^2-4\alpha(1+n) F\left[1+n,1,\alpha  t\right]^2\le 0.$$
Then after some straightforward computations we get $$k(t)=-4 \alpha^2 n t F[1+n,2,\alpha t] (F[1+n,1,\alpha t]+n F[1+n,2,\alpha t])$$ which is certainly a non-positive function.
This finishes the proof of the lemma.
\end{proof}

Now define $$J(s)=\frac{\alpha}{\pi}\int_{A_t} u_n(z) dxdy,$$ where $$u_n=\frac{|f^{(n)}(z)|^2e^{-\alpha |z|^2}}{\alpha^n n ! L_n(-\alpha |z|^2)}.$$
Now instead of \eqref{delu}, we use $-t\Delta u\le  4t\alpha(1+n)$, and instead of \eqref{iprim} we have $$J_n(s) \le (1-e^{-m s}),$$ provided that Conjecture~\ref{c3} is correct. This finishes the proof of Theorem~\ref{111} for the case $n\ge 2$, up to Lemma~\ref{lema} below. The sharpness part can be proved in a similar way as in the case $n=1$, but using, in this case, the equation \eqref{sharp} below.

\begin{lemma}\label{lema}
For $f\in \mathcal{F}_\alpha^2$, we have the following sharp inequality
$$I_n(f):=\frac{\alpha}{\pi}\int_{\mathbb{C}} \frac{|f^{(n)}(z)|^2e^{-\alpha |z|^2}}{\alpha^n n ! L_n(- \alpha |z|^2)}dxdy \le \|f\|^2_{2,\alpha},$$ provided that Conjecture~\ref{c3} is true.
\end{lemma}
\begin{proof} Since
  $$f^{(n)}(w) = \sum_{k=n}^\infty \frac{\Gamma(k+1)}{\Gamma(k-n+1)} a_k w^{k-n},$$ and because the polynomials $w^n$ and $w^m$ are orthogonal the Fock space  for $n\neq m$, it is clear that it is enough to prove the inequality for such functions. So we need to show that $$I_n(w^k)\le \|w^k\|^2_{2,\alpha}.$$
Let $k\ge n$ and $$a_k=\frac{I_n(w^k)} {\|w^k\|^2_{2,\alpha}}.$$ Then after straightforward computation we get

$$I_n(w^k)=\frac{a^{-k}\Gamma^2(1+k)}{ n! \Gamma^2(1+k-n)}\int_{0}^\infty \frac{e^{-s} s^{k-n}}{L_n(-s)}ds $$ and $$\|w^k\|^2_{2,\alpha}=\alpha^{-k} k!.$$
Thus
$$a_k=\frac{ \Gamma[1+k]}{n!\Gamma[1+k-n]^2}\int_0^\infty \frac{e^{-r}r^{k-n}}{L_n(- r)} dr$$
We need to show that
\begin{enumerate}
\item $\lim_{k\to \infty} a_k=1$.
\item $a_k<1$ and
\end{enumerate}
We will prove the first assertion, while the second is proved for few small integers.
This is the content of the following lemma.
\end{proof}

\begin{lemma}\label{lema4}
Let $p\ge n$ and $$A(p)=\frac{ \Gamma[1+p]}{n!\Gamma[1+p-n]^2}\int_0^\infty \frac{e^{-r}r^{p-n}}{L_n(-r)} dr,$$
where $L_n(-r)$ is the Laguerre polynomial.
Then \begin{equation}\label{sharp}\lim_{p\to \infty} A(p)=1,\end{equation} and if $n\le 4$, then $A(p)<1$.

\end{lemma}
\begin{proof}
First of all we have $$L_n(-r) = \sum_{m=0}^n \frac{\Gamma(1+n)}{\Gamma(1+m)^2 \Gamma(1+n-m)}r^m.$$ Fix $k$ and let $p=k+n$.  It follows from the definition of $A(h+k)$, $h\in\{0,\dots,n\}$  that \begin{equation}\label{sum}\sum_{h=0}^{n}\frac{A(h+k)\Gamma(1+h+k-n)^2}{\Gamma(1+h)^2\Gamma(1+h+k)\Gamma(1+n-h)}=\frac{\Gamma(1+k-n)}{(n!)^2}.\end{equation}
In order to prove $\lim_{k\to \infty} A(k+n)=1$, observe that, in view of \eqref{sum}, for every $n$ we have
\begin{equation}\label{nrecu}A(k+n)=C(k,n)-\sum_{h=0}^{n-1}C_h(k,n) A(k+h),\end{equation}
where
$$C(k,n)=\frac{(k-n)! (k+n)!}{(k!)^2}$$
and
$$C_h(k,n)=\frac{((h+k-n)!)^2 (n!)^2 (k+n)!}{(h!)^2 (k!)^2 (h+k)! (-h+n)!}.$$
Then we easily conclude that $$A(k+n)<C(k,n)\le \frac{(2n)!}{(n!)^2}.$$ So $A(k+n)$ is bounded for fixed $n$. The same hold for $A(k+h)$ for $h\le n$. Namely for $k\ge 2n$, $$A(k+h)\le \frac{(2h)!}{(h!)^2}\le \frac{(2n)!}{(n!)^2}.$$ Further we have $$\lim_{k\to \infty} C(k,n)=1$$ and $$\lim_{k\to\infty}C_h(k,n)=0.$$ This and \eqref{nrecu} implies that $$\lim_{k\to\infty}A(k+n)=1.$$
Let us now prove $A(k)<1$ for first few integers $n$.
For $n=2$ we have
$$A(2+k)= \frac{(2+k) \left(-k+k^3-2 (1+k) A(k)-4 (-1+k)^2 A(1+k)\right)}{(-1+k)^2 k^2}.$$
Let $$B(x,y)= \frac{(2+k) \left(-k+k^3-2 (1+k) x-4 (-1+k)^2 y\right)}{(-1+k)^2 k^2}.$$
Further $A(2)=0.427393$, $A(3)=0.662691$, $A(4)=0.784876$, $A(5)=0.852841$, $A(6)=0.893707$, $A(7)=0.919951$. So we get $A(j)<1$ for $j\le 7$ and $a(j)>j/(j+1)$ for $4\le j\le 7$.  Now we use mathematical induction to prove that $ \frac{m}{m+1}\le A(m)<1$ for $m>  7$. Let $m=k>7$ and assume that  $\frac{m}{m+1}\le A(m)<1$ and $\frac{m+1}{m+2}\le A(m+1)<1$. Then $A(m+2)\ge B(1,1)$, and $B(1,1)\ge \frac{m+2}{m+2+1} $ iff  $$(-1 + m)^2 m^2 (2 + m) (3 + m) (-18 + m (9 + (-8 + m) m))\ge 0,$$ and the last inequality is true for every $m\ge 8$. Moreover
\[\begin{split}A(m+2)&\le B(m/(m+1), (m+1)/(m+2))\\&=\frac{(2+m) \left(-3 m+m^3-\frac{4 (-1+m)^2 (1+m)}{2+m}\right)}{(-1+m)^2 m^2}<1,\end{split}\] because the last inequality is equivalent with $$(-1 + m)^2 m^2 (2 + m) > 0.$$

For $n=3$, we have \[\begin{split}A(3+k)&=\frac{(1+k) (2+k) (3+k)}{(-2+k) (-1+k) k}-\frac{(1+k) (2+k) (3+k)}{(-2+k) (-1+k) k}\\&\times\bigg(\frac{6A(k)}{ (-2+k) (-1+k) k}\\&+\frac{18(-2+k) A(1+k)}{ (-1+k) k (1+k)}+\frac{9(-2+k) (-1+k) A(2+k)}{ k (1+k) (2+k)}\bigg).\end{split}\] By proceeding similarly as above, we can prove that $\frac{k}{k+1}\le A(k)<1$ for $k\ge 14$, and that $A(k)< 1$ for $k\in\{1,\dots, 13\}$.

For $n=4$, we have \begin{equation}\label{lh}\begin{split}A(4+k)&=X-\bigg(\frac{24 X A(k)}{(-3+k) (-2+k) (-1+k) k}\\&+\frac{96 X (-3+k) A(1+k)}{(-2+k) (-1+k) k (1+k)}+\frac{72 X (-3+k) (-2+k) A(2+k)}{(-1+k) k (1+k) (2+k)}\\&+\frac{16 X(-3+k) (-2+k) (-1+k) A(3+k)}{k (1+k) (2+k) (3+k)}\bigg),\end{split}\end{equation} where $$X=\frac{ (1+k) (2+k) (3+k) (4+k)}{(-3+k) (-2+k) (-1+k) k}.$$

Then, by using Mathematica software  we obtain $A(j)<j/(j+1)$, $j=1,\dots, 16$ but $A(j)>j/(j+1)$ for $j=17,\dots, 27$.
By induction, and previous formula we obtain that  $j/(j+1)<A(j)<1$, for $j>27$. Namely, LHS of \eqref{lh}, for $A(j)\equiv 1$ is equal to
$\frac{(4+k) (-13680+k (11316+k (-5464+k (1921+k (-520+k (106+(-16+k) k))))))}{(-3+k)^2 (-2+k)^2 (-1+k)^2 k^2}$
and the last expression is bigger than $(k+4)/(k+5)$ if and only if $k\ge 28$. Further  the LHS of \eqref{lh}, for $A(j)= j/(j+1)$ is equal  to
$\frac{-32832+k (14544+k (-2916+k (-180+k (433+k (-216+k (58+(-12+k) k))))))}{(-3+k)^2 (-2+k)^2 (-1+k)^2 k^2}$
and the last expression is smaller than $1$ for every $k$.

\end{proof}

\begin{remark}
Similar to the proof of the Lemma~\ref{lema}, one can proceed to prove the statement for some constants greater than 4, but this still does not solve the general problem. We believe that the recurrent formula \eqref{nrecu} is suitable for proving the general case, and we also believe that \eqref{nrecu} implies $j/(j+1)<A(j)<1$ for $j> n^2$. This and some good estimations of $A(j)$ for $j\le n^2$ would conclude the proof of the general case, but we fail to do it.
\end{remark}


\subsection*{Acknowledgments} I would like to thank Dr. P. Melentijevi\'c for helpful comments regarding this paper.

\end{document}